\documentclass[graybox]{svmult}

\usepackage{helvet}         
\usepackage{courier}        
\usepackage{type1cm}        
\usepackage{makeidx}         
\usepackage{graphicx}        
\usepackage{multicol}        
\usepackage[bottom]{footmisc}
\usepackage{bm}
\usepackage{mathtools}

\makeindex             

\usepackage{latexsym}
\usepackage{dsfont}
\usepackage{bbm}
\usepackage{amssymb}
\usepackage{amsmath}

\newtheorem{Theorem}{Theorem}[section]
\newtheorem{Lemma}[Theorem]{Lemma}

\newtheorem{Prop}[Theorem]{Proposition}
\newtheorem{Rem}[Theorem]{Remark}

\usepackage[top=4.cm, bottom=4cm, outer=4cm, inner=3.8cm]{geometry}

\def\cF{\mathcal{F}}

   \def\fm{\mathfrak{m}}

\def\Erw{\mathbb{E}}

\def\N{\mathbb{N}}

\def\Prob{\mathbb{P}}

\def\R{\mathbb{R}}

\def\eps{\varepsilon}

\def\1{\vec{1}}
\def\3{{\ss}}

\def\RA{\Rightarrow}

\def\wh{\widehat}

\def\ovl{\overline}

\def\sg{\sigma^{>}}
\def\sgn{\sigma_{n}^{>}}

\def\sl{\sigma^{<}}
\def\sln{\sigma_{n}^{<}}

\def\supp{{\rm supp}}
\def\ov{\overline}
\def\8{\infty}
\def\wt{\widetilde}

\allowdisplaybreaks[4] 

\begin{document}

\title*{Null recurrence and transience of random difference equations in the contractive case}
\titlerunning{Null recurrence and transience of RDE in the contractive case}
\author{Gerold Alsmeyer, Dariusz Buraczewski and Alexander Iksanov}
\institute{Gerold Alsmeyer \at Institute of Mathematical Stochastics, Department
of Mathematics and Computer Science, University of M\"unster,
Einsteinstrasse 62, D-48149 M\"unster, Germany.\at
\email{gerolda@math.uni-muenster.de}\\
Dariusz Buraczewski \at Institute of Mathematics, University of Wroclaw,
pl. Grunwaldzki 2/4, 50-384 Wroclaw, Poland.\at
\email{dbura@math.uni.wroc.pl}\\
Alexander Iksanov \at Faculty of Computer Science and Cybernetics, Taras Shevchenko National University of Kyiv, 01601 Kyiv, Ukraine and Institute of Mathematics, University of Wroclaw, pl. Grunwaldzki 2/4, 50-384 Wroclaw, Poland. \at
\email{iksan@univ.kiev.ua}}

\maketitle

\abstract{Given a sequence $(M_{k}, Q_{k})_{k\ge 1}$ of independent, identically distributed ran\-dom vectors with nonnegative components, we consider the recursive Markov chain $(X_{n})_{n\ge 0}$, defined by the random difference equation $X_{n}=M_{n}X_{n-1}+Q_{n}$ for $n\ge 1$, where $X_{0}$ is independent of $(M_{k}, Q_{k})_{k\ge 1}$. Criteria for the null recurrence/transience are provided in the situation where $(X_{n})_{n\ge 0}$ is contractive in the sense that $M_{1}\cdot\ldots\cdot M_{n}\to 0$ a.s., yet occasional large values of the $Q_{n}$ overcompensate the contractive behavior so that positive recurrence fails to hold. We also investigate the attractor set of $(X_{n})_{n\ge 0}$ under the sole assumption that this chain is locally contractive and recurrent.}

\bigskip



\section{Introduction}\label{sec:intro}

Let $(M_{n}, Q_{n})_{n\ge 1}$ be a sequence of independent, identically distributed (iid) $\R_+^{2}$-valued random vectors with common law $\mu$ and generic copy $(M, Q)$, where $\R_+:=[0,\infty)$. Further, let $X_{0}$ be a nonnegative random variable which is independent of $(M_{n},Q_{n})_{n\ge 1}$. Then the sequence $(X_{n})_{n\ge 0}$, recursively defined by the random difference equation (RDE)
\begin{equation}\label{chain}
X_{n}\ :=\ M_{n}X_{n-1}+Q_{n},\quad n\ge 1,
\end{equation}
forms a temporally homogeneous Markov chain with transition kernel $P$ given by
$$ Pf(x)\ =\ \int f(mx+q)\ {\rm d}\mu(m,q) $$
for bounded measurable functions $f:\R\to\R$. The operator $P$ is Feller because it maps bounded continuous $f$ to functions of the same type. To underline the role of the starting point we occasionally write $X_{n}^{x}$ when $X_{0} = x$ a.s. Since $M$, $Q$ and $X_{0}$ are nonnegative, $(X_{n})_{n\ge 0}$ has state space $\R_{+}$.

\vspace{.1cm}
The sequence $(X_{n})_{n\ge 0}$ may also be viewed as a \emph{forward iterated function system}, viz.
$$ X_{n}\ =\ \Psi_{n}(X_{n-1})\ =\ \Psi_{n}\circ\ldots\circ\Psi_{1}(X_{0}),\quad n\ge 1, $$
where $\Psi_{n}(t):=Q_{n}+M_{n} t$ for $n\ge 1$ and $\circ$ denotes composition, and thus opposed to its closely related counterpart
of \emph{backward iterations}
$$\wh{X}_{0}\ :=\ X_{0}\quad\text{and}\quad \wh{X}_{n}\ :=\ \Psi_{1}\circ\ldots\circ\Psi_{n}(X_{0}),\quad n\ge 1. $$
The relation is established by the obvious fact that $X_{n}$ has the same law as $\wh{X}_{n}$ for each $n$, regardless of the law of $X_{0}$.

\vspace{.1cm}
Put
$$\Pi_{0}\ :=\ 1\quad\text{and}\quad\Pi_{n}\ :=\ M_{1}M_{2}\cdot\ldots\cdot M_{n},\quad n\ge 1.$$
Assuming that
\begin{equation}\label{trivial}
\Prob(M=0)\ =\ 0\quad \text{and}\quad\Prob (Q=0)\ <\ 1
\end{equation}
and
\begin{equation}\label{trivial2}
\Prob(Mr+Q=r)\ <\ 1\quad\text{for all }r\ge 0,
\end{equation}
Goldie and Maller \cite[Theorem 2.1]{GolMal:00} showed (actually, these authors did not assume that $M$ and $Q$ are nonnegative) that the series
$\sum_{k\ge 1}\Pi_{k-1}Q_{k}$, called \emph{perpetuity}, is a.s.\ convergent provided that
\begin{equation}\label{33}
\lim_{n\to\infty}\Pi_{n}\ =\ 0\quad\text{a.s.\quad and}\quad I_{Q}\ :=\ \int_{(1,\,\infty)}J_{-}(x)\ \Prob(\log Q\in {\rm d}x)\ <\ \infty,
\end{equation}
where
\begin{equation}\label{jx}
J_{-}(y):=\frac{y}{\Erw(y\wedge\log_-M)},\quad y>0
\end{equation}
and $\log_- x=-\min(\log x, 0)$. Equivalently, the Markov chain $(X_{n})_{n\ge 0}$ is then positive recurrent with unique invariant distribution given by the law of the perpetuity. It is also well-known what happens in the ``trivial cases'' when at least one of the conditions \eqref{trivial} and \eqref{trivial2} fails \cite[Theorem 3.1]{GolMal:00}:
\begin{description}[(b)]
\item[(a)] If $\Prob(M=0)>0$, then $\tau:=\inf\{k\ge 1:M_{k}=0\}$ is a.s. finite, and the perpetuity trivially converges to the a.s.\ finite random variable $\sum_{k=1}^{\tau}\Pi_{k-1}Q_{k}$, its law being the unique invariant distribution of $(X_{n})_{n\ge 0}$.\vspace{.1cm}
\item[(b)] If $\Prob(Q=0)=1$, then $\sum_{k\ge 1}\Pi_{k-1}Q_{k}=0$ a.s.\vspace{.1cm}
\item[(c)] If $\Prob(Q+Mr=r)=1$ for some $r\ge 0$ and $\Prob(M=0)=0$, then either $\delta_{r}$, the Dirac measure at $r$, is the unique invariant distribution of $(X_{n})_{n\ge 1}$, or every distribution is invariant.
\end{description}
Further information on RDE and perpetuities can be found in the recent books \cite{BurDamMik:16} and \cite{Iksanov:17}.

\vspace{.1cm}
If \eqref{trivial}, \eqref{trivial2},
\begin{equation}\label{30}
\lim_{n\to\infty}\Pi_{n}=0\quad\text{a.s.\quad and}\quad I_{Q}\,=\,\infty
\end{equation}
hold, which are assumptions in most of our results hereafter (with the exception of Section \ref{attr}) and particularly satisfied if
\begin{equation}\label{eq:log}
-\infty\ \le\ \Erw\log M\ <\ 0\quad\text{and}\quad\Erw\log_{+} Q\ =\ \infty,
\end{equation}
where $\log_+ x=\max(\log x, 0)$, then the afore-stated result \cite[Theorem 2.1]{GolMal:00} by Goldie and Maller implies that $(X_{n})_{n\ge 0}$ must be either null recurrent or transient. Our purpose is to provide conditions for each of these alternatives and also to investigate the path behavior of $(X_{n})_{n\ge 0}$. We refer to \eqref{30} as the \emph{divergent contractive case} because, on the one hand, $\Pi_{n}\to 0$ a.s. still renders $\Psi_{n}\circ\ldots\circ\Psi_{1}$ to be contractions for sufficiently large $n$, while, on the other hand, $I_Q=\infty$ entails that occasional large values of the $Q_{n}$ overcompensate this contractive behavior in such a way that positive recurrence does no longer hold. As a consequence, $\sum_{k\ge 1}\Pi_{k-1}Q_{k}=\infty$ a.s. and so the backward iterations $\wh{X}_{n}=\Pi_{n}X_{0}+\sum_{k=1}^{n}\Pi_{k-1}Q_{k}$ diverge to $\infty$ a.s. regardless of whether the chain $(X_{n})_{n\ge 0}$ is null recurrent or transient. The question of which alternative occurs relies on a delicate interplay between the $\Pi_{n}$ and the $Q_{n}$. Our main results (Theorems \ref{main11} and \ref{main12}), for simplicity here confined to the situation when \eqref{trivial}, \eqref{trivial2}, \eqref{eq:log} hold and $s:=\lim_{t\to\infty}t\,\Prob(\log Q>t)$ exists, assert that $(X_{n})_{n\ge 0}$ is null recurrent if $s<-\Erw\log M$ and transient if $s>-\Erw\log M$. For deterministic $M\in (0,1)$, i.e., autoregressive sequences $(X_{n})_{n\ge 0}$, this result goes already back to Kellerer \cite[Theorem 3.1]{Kellerer:92} and was later also proved by Zeevi and Glynn \cite[Theorem 1]{ZeeviGlynn:04}, though under a further extra assumption, namely that $Q$ has log-Cauchy tails with scale parameter $s$, i.e.
$$ \Prob(\log(1+Q)>t)\ =\ \frac{1}{1+st}\quad\text{for all }t>0. $$
On the other hand, they could show null recurrence of $(X_{n})_{n\ge 0}$ even in the boundary case $s=-\log M$. Kellerer's result will be of some relevance here because we will take advantage of it in combination with a stochastic comparison technique (see Section \ref{M<=gamma<1}, in particular Proposition \ref{Kellerer's result}). Finally, we mention work by Bauernschubert \cite{Bauernschubert:13}, Buraczewski and Iksanov \cite{Buraczewski+Iksanov:15}, Pakes \cite{Pakes:83} and, most recently, by Zerner \cite{Zerner:2016+} on the divergent contractive case, yet only the last one studies the recurrence problem and is in fact close to our work. We will therefore comment on the connections in more detail in Remark \ref{zer}.

\vspace{.1cm}
In the critical case $\Erw\log M=0$ not studied here, when $\limsup_{n\to\infty}\Pi_{n}=\infty$ a.s. and thus non-contraction holds, a sufficient criterion for the null recurrence of $(X_{n})_{n\ge 0}$ and the existence of an essentially unique invariant Radon measure $\nu$ was given by Babillot et al. \cite{BabBouElie:97}, namely
$$ \Erw|\log M|^{2+\delta}\,<\,\infty\quad\text{and}\quad\Erw(\log_{+}Q)^{2+\delta}\,<\,\infty\quad\text{for some }\delta>0. $$
For other aspects like the tail behavior of $\nu$ or the convergence $\wh{X}_{n}$ after suitable normalization see \cite{Brofferio:03, Buraczewski:07, Grincev:76, HitczenkoWes:11,Iksanov+Pilipenko+Samoilenko:17,RachSamo:95}.

\vspace{.1cm}
The paper is organized as follows. In Section \ref{sec:background}, we review known results about general locally contractive Markov chains which form the theoretical basis of the present work. Our main results are stated in Section \ref{results} and proved in Sections \ref{M<=gamma<1}, \ref{sec:tail lemma} and \ref{sec:main11 and main12}. In Section \ref{attr} we investigate the attractor set of the Markov chain $(X_{n})_{n\ge 0}$ under the sole assumption that $(X_{n})_{n\ge 0}$ is locally contractive and recurrent.

\section{Theoretical background}\label{sec:background}

We start by giving some useful necessary and sufficient conditions for the transience and recurrence of the sequence $(X_{n})_{n\ge 0}$. The following definition plays a fundamental role in the critical case $\Erw\log M=0$, see \cite{BabBouElie:97,Benda:98b,Brofferio:03,BroBura:13,Buraczewski:07,PeigneWoess:11a}. A general Markov chain $(X_{n})_{n\ge 0}$, possibly taking values of both signs, is called \emph{locally contractive} if, for any compact set $K$ and all $x,y\in\R$,
\begin{equation}\label{eq: local contraction}
\lim_{n\to\infty}\big| X_{n}^{x}-X_{n}^{y}\big| \cdot\1_{\{X_{n}^{x}\in K \}}\ =\ 0\quad\text{a.s.}
\end{equation}
For the chain $(X_{n})_{n\ge 0}$ to be studied here, we observe that, under \eqref{30},
$$ \big| X_{n}^{x}-X_{n}^{y} \big|\ =\ \Pi_{n}|x-y|\ \underset{n\to\infty}{\longrightarrow} \ 0\quad\text{ a.s.} $$
for all $x,y\in\R$. This means that $(X_{n})_{n\ge 0}$ is contractive and hence locally contractive. Yet, it may hold that
$$ \Prob\left(\lim_{n\to\infty}|X_{n}^{x}-x|=\infty\right)\ =\ 1$$
for any $x\in\R$ in which case the chain is called \emph{transient}. We quote the following result from \cite[Lemma 2.2]{PeigneWoess:11a}.

\begin{Lemma}\label{lem:1}
If $(X_{n})_{n\ge 0}$ is locally contractive, then the following dichotomy holds: either
\begin{equation}\label{eq:3}
\Prob\left(\lim_{n\to\infty}|X_{n}^{x}-x|=\infty\right)\ =\ 0\quad\text{for all }x\in\R
\end{equation}
or
\begin{equation}\label{eq:4}
\Prob\left(\lim_{n\to\infty}|X_{n}^{x}-x|=\infty\right)\ =\ 1\quad\text{for all }x\in\R.
\end{equation}
\end{Lemma}

The lemma states that either $(X_{n})_{n\ge 0}$ is transient or visits a large interval infinitely often (i.o.). The Markov chain $(X_{n})_{n\ge 0}$ is called \emph{recurrent} if there exists a nonempty closed set $L\subset\R$ such that $\Prob(X_{n}^{x}\in U\text{ i.o.})=1$ for every $x\in L$
and every open set $U$ that intersects $L$. Plainly, recurrence is a local property of the path of $(X_{n})_{n\ge 0}$.

\vspace{.1cm}
The next lemma can be found in \cite[Theorem 3.8]{Benda:98b} and \cite[Theorem 2.13]{PeigneWoess:11a}.

\begin{Lemma}\label{lem:3}
If $(X_{n})_{n\ge 0}$ is locally contractive and recurrent, then there exists a unique (up to a multiplicative constant) invariant locally finite measure $\nu$.
\end{Lemma}

The Markov chain $(X_{n})_{n\ge 0}$ is called \emph{positive recurrent} if $\nu(L)<\infty$ and \emph{null recurrent}, otherwise.

\vspace{.1cm}
Our third lemma was stated as Proposition 1.3 in \cite{Benda:98b}. Since this report has never been published, we present a short proof.

\begin{Lemma}\label{lem:2}
Let $(X_{n})_{n\ge 0}$ be a locally contractive Markov chain and $U$ an open subset of $\R$. Then $\Prob(X_{n}^{x}\in U~{\rm i.o.})<1$ for some $x\in\R$ implies $\sum_{n\ge 0} \Prob(X_{n}^{y}\in K)<\infty$ for all $y\in\R$ and all compact $K\subset U$.
\end{Lemma}

\begin{proof}
Take $x$ such that $\Prob(X_{n}^{x}\in U\ \text{i.o.})<1$. Then there exists $n_{1}\in\N$ such that
$$ \Prob(X_{n}^{x}\notin U\text{ for all }n\ge n_{1})\ >\ 0. $$
Now fix an arbitrary $y\in \R$ and a compact $K\subset U$. Defining the compact set $K_{y}:=K \cup\{y\}$, the local contractivity implies that for some $n_{2}\in\N$
\begin{equation}\label{star}
\Prob\left(X_{n}^{z}\notin K\text{ for all } n\ge n_{2}\ \text{ and some } z\in K_{y}\right)\ =:\ \delta\ >\ 0.
\end{equation}
For $z\in K_{y}$, consider the sequence of stopping times
\begin{align*}
T_{0}^{z}\ &=\ 0\quad\text{and}\quad T_{n}^{z}\ =\ \inf\{ k> T_{n-1}^{z}:\; X_{k}^{z}\in K\}\quad\text{for }n\ge 1.
\end{align*}
Then \eqref{star} implies that $\Prob(T^{z}_{n_{2}}<\infty)\le 1-\delta$ for each $z\in K_{y}$. Consequently,
$$ \Prob\left(T^{y}_{nn_{2}}<\infty\right)\ \le\ (1-\delta) \Prob\left(T_{(n-1)n_{2}}^{y}<\infty\right)\ \le\ (1-\delta)^{n} $$
for all $n\ge 1$ and thus
\begin{align*}
\sum_{n\ge 0} \Prob(X_{n}^{y} \in K)\ &=\ \Erw\left(\sum_{n\ge 0}\1_{\{X_{n}^{y}\in K\}}\right)\ \le\ \sum_{n\ge 0}\Prob\left(T_{n}^{y} <\infty\right)\\
&\le\ \sum_{n\ge 0} n_{2}\,\Prob\left(T_{n n_{2}}^{y} <\infty\right)\ \le\ n_{2}/\delta\ <\ \infty.\tag*\qed
\end{align*}
\end{proof}

A combination of Lemmata \ref{lem:1} and \ref{lem:2} provides us with

\begin{Prop}\label{transient}
For a locally contractive Markov chain $(X_{n})_{n\ge 0}$ on $\R$, the following assertions are equivalent:
\begin{description}[(b)]
\item[(a)] The chain is transient.\vspace{.1cm}
\item[(b)] $\lim_{n\to\infty}|X_{n}^{x}|=\infty$ a.s. for all $x\in\R$.\vspace{.1cm}
\item[(c)] $\Prob(X_{n}^{x}\in U~{\rm i.o.})< 1$ for any bounded open $U\subset\R$ and some/all $x\in\R$.\vspace{.1cm}
\item[(d)] $\sum_{n\ge 0}\Prob(X_{n}^{x}\in K)<\infty$ for any compact $K\subset\R$ and some/all $x\in\R$.
\end{description}
\end{Prop}

\begin{proof}
The equivalence of (a), (b) and (c) is obvious. By Lemma \ref{lem:2}, (c) entails (d), while the Borel-Cantelli lemma gives the converse.\qed
\end{proof}

Now we consider the case when \eqref{eq:4} is satisfied. For any $\omega$, we define $L^{x}(\omega)$ to be the set of accumulation points of $(X_{n}^x(\omega))_{n\ge 0}$, i.e.
$$ L^{x}(\omega)\ :=\ \bigcap_{m\ge 1}\overline{\{X_{n}^x(\omega):n\ge m\}}, $$
where $\overline{C}$ denotes the closure of a set $C$. It is known \cite{Benda:98b,PeigneWoess:11a} that $L^x(\omega)$ does not depend on $x$ and $\omega$. In fact, there exists a deterministic set $L\subset\R$ (called the \emph{attractor set} or \emph{limit set}) such that
$$ \Prob\{L^x(\cdot) = L \ \text{ for all } x\in \R\}=1. $$

\begin{Prop}\label{recur}
For a locally contractive Markov chain $(X_{n})_{n\ge 0}$ on $\R$, the following assertions are equivalent:
\begin{description}[(b)]
\item[(a)] The chain is recurrent.\vspace{.1cm}
\item[(b)] $\liminf_{n\to\infty}|X_{n}^x-x|<\infty$ a.s. for all $x\in\R$.\vspace{.1cm}
\item[(c)] $\liminf_{n\to\infty}|X_{n}|<\infty$ a.s.\vspace{.1cm}
\item[(d)] $\sum_{n\ge 0}\Prob\{X_{n}^{x}\in K\}=\infty$ for a nonempty compact set $K$ and some/all $x\in\R$.
\end{description}
\end{Prop}

\begin{proof}
In view of the contrapositive Proposition \ref{transient}, we must only verify for ``(a)$\RA$(d)'' that the sum in (d) is indeed infinite for some compact $K\ne\oslash$ and \emph{all} $x\in\R$. W.l.o.g. let $K=[-2b,2b]$ for some $b>0$ and $y\in\R$ such that, by (a), $\sum_{n\ge 0}\1_{\{|X_{n}^{y}|\le b\}}=\infty$ a.s.\ and thus $\sum_{n\ge 0}\Prob(|X_{n}^{y}|\le b)=\infty$. Local contractivity implies that $\sigma_{x}:=\sup\{n\ge 0:|X_{n}^{x}-X_{n}^{y}|>b,\,|X_{n}^{y}|\le b\}$ is a.s.\ finite for \emph{all} $x\in\R$. Consequently, $X_{n}^{x}$ hits $[-2b,2b]$ whenever $X_{n}^{y}$ hits $[-b,b]$ for $n>\sigma_{x}$, in particular $\sum_{n\ge 0}\1_{\{|X_{n}^{x}|\le 2b\}}=\infty$ a.s.\ and thus $\sum_{n\ge 0}\Prob(|X_{n}^{x}|\le 2b)=\infty$ for all $x\in\R$.\qed
\end{proof}

\section{Results}\label{results}

In order to formulate the main result, we need
\begin{align}
\begin{split}\label{eq:def of s_* and s^*}
s_{*}\ :=\ &\liminf_{t\to\infty}\,t\,\Prob(\log Q>t),\\
s^{*}\ :=\ &\limsup_{t\to\infty}\,t\,\Prob(\log Q>t)
\end{split}
\end{align}
for which $0\le s_{*}\le s^{*}\le\infty$ holds true. In some places, the condition
\begin{equation}\label{tail2}
\lim_{t\to\infty}t\,\Prob(\log Q>t)\ =\ s\ \in\ [0,\infty]
\end{equation}
will be used. Finally, put $\fm^{\pm}:=\Erw\log_{\pm}M$ and, if $\fm^{+}\wedge\fm^{-}<\infty$,
$$ \fm\ :=\ \Erw\log M\ =\ \fm^{+}-\fm^{-} $$
which is then in $[-\infty,0)$ by our standing assumption $\Pi_{n}\to 0$ a.s.

\begin{Theorem}\label{main11}
Let $\fm\in [-\infty,0)$ and \eqref{trivial}, \eqref{trivial2}, \eqref{30} be valid. Then the following assertions hold:
\begin{description}[(b)]\itemsep2pt
\item[(a)] $(X_{n})_{n\ge 0}$ is null recurrent if $s^{*}<-\fm$.
\item[(b)] $(X_{n})_{n\ge 0}$ is transient if $s_{*}>-\fm\ ($thus $\fm>-\infty)$.
\end{description}
\end{Theorem}

\begin{Rem}\label{zer}\rm
In the recent paper \cite{Zerner:2016+}, Zerner studies the recurrence/transience of $(X_{n})_{n\ge 0}$ defined by \eqref{chain} in the more general setting when $M$ is a nonnegative $d\times d$ random matrix and $Q$ an $\R_{+}^{d}$-valued random vector. A specialization of his Theorem 5 to the one-dimensional case $d=1$ reads as follows. Suppose that
\begin{equation}\label{boun}
M\in [a,b]\quad\text{a.s.}
\end{equation}
for some $0<a<b<\infty$ and that either $\lim_{t\to\infty}\,t^\beta\,\Prob(\log Q>t)=0$ for some $\beta\in (2/3,1)$, or $s_{*}>-\fm$. Let $y\in (0,\infty)$ be such that $\Prob(Q\le y)>0$. Then $(X_{n})_{n\ge 0}$ is recurrent if, and only if,
\begin{equation}\label{eq:Zerner condition}
\sum_{n\ge 0}\prod_{k=0}^{n}\Prob(Q\le ye^{-k\fm})\ =\ \infty.
\end{equation}
It is not difficult to verify that \eqref{eq:Zerner condition} holds if $s^{*}<-\fm$ and that it fails if $s_{*}>-\fm$. Therefore, Zerner's result contains our Theorem \ref{main11} under the additional assumption \eqref{boun}.
\end{Rem}

\begin{Rem}\label{Lyapunov functions}\rm
If $\log M$ and $\log Q$ are both integrable and $D(x):=\log X_{1}^{x}-\log X_{0}^{x}$, then
$$ \Erw D(x)\ =\ \Erw\log(M+x^{-1}Q)\ \xrightarrow{x\to\infty}\ 0 $$
shows that $(\log X_{n})_{n\ge 0}$ forms a Markov chain with asymptotic drift zero. Such  chains are studied at length by Denisov, Korshunov and Wachtel in a recent monograph-like publication \cite{DenKorWach:16}. They also provide conditions for recurrence and transience in terms of truncated moments of $D(x)$, see their Corollaries 2.11 and 2.16, but these appear to be more complicated and more restrictive than ours.
\end{Rem}

\begin{Rem}\rm
Here is a comment on the boundary case $s=-\fm$ not covered by Theorem \ref{main11}. Assuming $M=e^\fm$ a.s., it can be shown that the null recurrence/transience of $(X_{n})_{n\ge 0}$ is equivalent to the
divergence/convergence of the series
$$ \sum_{n\ge 1}\Prob\left(\max_{1\le k\le n}\,e^{\fm (k-1)}Q_{k}\le e^{x}\right)\ =\ \sum_{n\ge 1}\prod_{k=0}^{n-1} F(x-\fm k) $$
for some/all $x\ge 0$, where $F(y):=\Prob(\log Q\le y)$. Indeed, assuming $X_{0}=0$, the transience assertion follows when using $\Prob(X_{n}\le e^{x})\le \Prob(\max_{1\le k\le n}\,e^{\fm (k-1)}Q_{k}\le e^{x})$, while the null recurrence claim is shown by a thorough inspection and adjustment of the proof of Theorem \ref{main11}(a). Using Kummer's test as stated in \cite{Tong:94}, we then further conclude that $(X_{n})_{n\ge 0}$ is null recurrent if, and only if, there exist positive $p_{1},p_{2},\ldots$ such that $F(-\fm k)\ge p_{k}/p_{k+1}$ and $\sum_{k\ge
1}(1/p_{k})=\infty$. For applications, the following sufficient condition, which is a consequence of Bertrand's test \cite[p.~408]{Stromberg:81}, may be more convenient. If
$$ 1-F(x)\ =\ \frac{-\fm}{x}+\frac{f(x)}{x\log x}, $$
then $(X_{n})_{n\ge 0}$ is null recurrent if $\displaystyle\limsup_{x\to\infty}f(x)<-\fm$, and transient if $\displaystyle\liminf_{x\to\infty}f(x)>-\fm$.
\end{Rem}

If $\fm^{-}=\fm^{+}=\infty$ and $s^{*}<\infty$, then $(X_{n})_{n\ge 0}$ is \emph{always} null recurrent as the next theorem will confirm. Its proof will be based on finding an appropriate subsequence of $(X_{n})_{n\ge 0}$ which satisfies the assumptions of Theorem \ref{main11}(a).

\begin{Theorem}\label{main12}
Let $\fm^{+}=\fm^{-}=\infty$, $s^\ast<\infty$ and \eqref{trivial}, \eqref{trivial2}, \eqref{30} be valid. Then $(X_{n})_{n\ge 0}$ is null recurrent.
\end{Theorem}

The two theorems are proved in Section \ref{sec:main11 and main12} after some preparatory work in Sections \ref{M<=gamma<1} and \ref{sec:tail lemma}.

\begin{Rem}\label{rem:main12}\rm
It is worthwhile to point out that the assumptions of the previous theorem impose some constraint on the tails of $\log_{+}M$. Namely, given these assumptions, the negative divergence of the random walk $S_{n}:=\log\Pi_{n}$, $n\ge 0$, that is $S_{n}\to-\infty$ a.s., entails
\begin{align*}
I_{M}\ =\ \int_{(1,\,\infty)}J_{-}(x)\ \Prob(\log M\in {\rm d}x)\ <\ \infty
\end{align*}
by Erickson's theorem \cite[Theorem 2]{Erickson:73}. But this in combination with $I_{Q}=\infty$ and $s^{*}<\infty$ further implies by stochastic comparison that
$$ r_{*}\ :=\ \liminf_{t\to\infty}t\,\Prob(\log M>t)\ =\ 0. $$
Indeed, if the latter failed to hold, i.e. $r_{*}>0$, then
$$ \Prob(\log M>t)\ \ge\ \frac{r_{*}}{2t}\ \ge\ \frac{r_{*}}{4s^{*}}\,\Prob(\log Q>t) $$
for all sufficiently large $t$, say $t\ge t_{0}$, which in turn would entail the contradiction
$$ I_{M}-J_-(0+)\ \ge\ \frac{r_{*}}{4s^{*}}\int_{(t_{0},\infty)}J_{-}'(x)\ \Prob(\log Q>x)\ {\rm d}x\ =\ \infty. $$
\end{Rem}

\section{The cases $M\le\gamma$ and $M\geq \gamma$: Two comparison lemmata and Kellerer's result}\label{M<=gamma<1}

This section collects some useful results for the cases when $M\le\gamma$ or $M\geq \gamma$ a.s.\ for a constant $\gamma\in (0,1)$, in particular Kellerer's unpublished recurrence result \cite{Kellerer:92} for this situation, see Proposition \ref{Kellerer's result} below. Whenever given iid nonnegative $Q_{1},Q_{2},\ldots$ with generic copy $Q$, let $(X_{n}(\gamma))_{n\ge 0}$ be defined by
$$ X_{n}(\gamma)\ =\ \gamma X_{n-1}(\gamma)+Q_{n},\quad n\ge 1,$$ where $X_{0}(\gamma)$ is independent of $(Q_{n})_{n\ge 1}$.
We start with two comparison lemmata which treat two RDE with identical $M\le\gamma$ but different $Q$.

\begin{Lemma}\label{comparison lemma}
Let $(M_{n},Q_{n},Q_{n}')_{n\ge 1}$ be a sequence of iid random vectors with nonnegative components and generic copy $(M,Q,Q')$ such that $M\le\gamma$ a.s. for some $\gamma\in (0,1)$ and
\begin{equation}\label{eq:tail condition Q'}
\Prob(Q'>t)\ \ge\ \Prob(Q>t)
\end{equation}
for some $t_{0}\ge 0$ and all $t\ge t_{0}$. Define
$$ X_{n}\,:=\,M_{n}X_{n-1}+Q_{n}\quad\text{and}\quad X_{n}'\,:\,=M_{n}X_{n-1}'+Q_{n}' $$
for $n\ge 1$, where $X_{0}^\prime$ is independent of $X_{0}$ and $(M_{k}, Q_{k})_{k\ge 1}$. Then
\begin{align*}
(X_{n})_{n\ge 0}\text{ transient}\quad\Longrightarrow\quad (X_{n}')_{n\ge 0}\text{ transient}\\
\shortintertext{or, equivalently,}
(X_{n}')_{n\ge 0}\text{ recurrent}\quad\Longrightarrow\quad (X_{n})_{n\ge 0}\text{ recurrent}
\end{align*}
\end{Lemma}

\begin{proof}
The tail condition \eqref{eq:tail condition Q'} ensures that we may choose a coupling $(Q,Q')$ such that $Q'\ge Q-t_{0}$ a.s. Then, with $(Q_{n},Q_{n}')_{n\ge 1}$ being iid copies of $(Q,Q')$, it follows that
\begin{align*}
X_{n}'-X_{n}\ &=\ M_{n}(X_{n-1}'-X_{n-1})\ +\ Q_{n}'-Q_{n}\\
&\ge\ M_{n}(X_{n-1}'-X_{n-1})\ -\ t_{0}\\
\ldots\ &\ge\ \left(\prod_{k=1}^{n}M_{k}\right)(X_{0}'-X_{0})\ -\ t_{0}\sum_{k=0}^{n-1}\gamma^{k}\quad\text{a.s.}
\shortintertext{and thereby}
&\liminf_{n\to\infty}\,(X_{n}'-X_{n})\ \ge\ -\frac{t_{0}}{1-\gamma}\quad\text{a.s.}
\end{align*}
which obviously proves the asserted implication.\qed
\end{proof}

\begin{Lemma}\label{comparison lemma 2}
Replace condition \eqref{eq:tail condition Q'} in Lemma \ref{comparison lemma} with
\begin{equation}
Q'\ =\ \1_{\{Q>\beta\}}Q
\end{equation}
for some $\beta>0$, thus $\Prob(Q'=0)=\Prob(Q\le\beta)$. Then
$$ (X_{n}')_{n\ge 0}\text{ recurrent}\quad\Longleftrightarrow\quad (X_{n})_{n\ge 0}\text{ recurrent}. $$
\end{Lemma}

\begin{proof}
Here it suffices to point out that
\begin{align*}
|X_{n}'-X_{n}|\ &=\ |M_{n}(X_{n-1}'-X_{n-1})\ +\ Q_{n}'-Q_{n}|\\
&\le\ M_{n}|X_{n-1}'-X_{n-1}|\ +\ \1_{\{Q_{n}\le\beta\}}Q_{n}\\
\ldots\ &\le\ \left(\prod_{k=1}^{n}M_{k}\right)(X_{0}'-X_{0})\ +\ \beta\sum_{k=0}^{n-1}\gamma^{k}\\
&\le\ \gamma^{n}(X_{0}'-X_{0})\ +\ \frac{\beta}{1-\gamma}\quad\text{a.s.}
\end{align*}
for all $n\ge 1$.\qed
\end{proof}

The announced result by Kellerer including its proof (with some minor modifications), taken from his unpublished Technical Report \cite[Theorem ~3.1]{Kellerer:92}, is given next.

\begin{Prop}\label{Kellerer's result}
Let $0<\gamma<1$. Then the following assertions hold:
\begin{description}
\item[(a)] $(X_{n})_{n\ge 0}$ is transient if $M\ge\gamma$ a.s. and
\begin{equation}\label{lower tail condition}
s_{*}\ =\ \liminf_{t\to\infty}t\,\Prob(\log Q>t)\ >\ \log(1/\gamma).
\end{equation}
\item[(b)] $(X_{n})_{n\ge 0}$ is null recurrent if $M\le\gamma$ a.s. and
\begin{equation}\label{upper tail condition}
s^{*}\ =\ \limsup_{t\to\infty}t\,\Prob(\log Q>t)\ <\ \log(1/\gamma).
\end{equation}
\end{description}
\end{Prop}

\begin{proof}
It is enough to consider (in both parts) the case when $M=\gamma$ a.s. and thus the Markov chain $(X_{n}(\gamma))_{n\ge 0}$ as defined above. We may further assume that $X_{0}(\gamma)=0$ and put $\theta:=\log(1/\gamma)$.

\vspace{.1cm}
\emph{Transience.} It suffices to show that $\sum_{n\ge 1}\Prob(\wh{X}_{n}(\gamma)\le e^{t})<\infty$ for all $t\ge 0$. Fixing $t$ and any $\eps>0$ with $(1+\eps)\theta<s_{*}$, pick $m\in\N$ so large that
$$ \inf_{k\ge m+1}k\theta\,\Prob(\log Q>t+k\theta)\ \ge \ (1+\eps)\theta. $$
Then we infer for all $n>m$
\begin{align*}
\Prob(\wh{X}_{n}(\gamma)\le e^{t})\ &=\ \Prob\left(\sum_{k=1}^{n}\gamma^{k-1}Q_{k}\le e^{t}\right)\\
&\le\ \Prob\big(\log Q_{k}\le t+(k-1)\theta,\,1\le k\le n\big)\\
&\le\ \prod_{k=m+1}^{n}\big(1-\Prob(\log Q>t+k\theta)\big)\\
&\le\ \prod_{k=m+1}^{n}\left(1-\frac{1+\eps}{k}\right)\\
&\le\ \prod_{k=m+1}^{n}\left(1-\frac{1}{k}\right)^{1+\eps}\ =\ \left(\frac{m}{n}\right)^{1+\eps}
\end{align*}
where $(1-x)^{1+\eps}\ge 1-(1+\eps)x$ for all $x\in [0,1]$ has been utilized for the last inequality. Consequently, $\sum_{n\ge 1}\Prob(\wh{X}_{n}(\gamma)\le e^{t})<\infty$, and the transience of $(X_{n}(\gamma))_{n\ge 0}$ follows by Proposition \ref{transient}.

\vspace{.2cm}
\emph{Null recurrence.}
By Lemma \ref{comparison lemma 2}, we may assume w.l.o.g. that, for some sufficiently small $\varepsilon>0$, $\delta\,:=\,\Prob(Q=0)\,\ge\,\gamma^{\eps}$ and
\begin{align*}
&\sup_{t\ge 1}t\,\Prob(\log Q>t)\ \le\ (1-\eps)\theta.
\end{align*}
Put also $m_{n}:=\theta^{-1}(m+\log n)$ for integer $m\ge 1$ so large that $$g(x,n):=(x-1)\theta-\log n\ge 1\vee (1-\varepsilon)\theta$$ for all $x\in (m_{n},\infty)$. Note that $\delta^{m_{n}}\ge (e^mn)^{-\eps}$. For all $n\ge 1$ so large that $g(n,n)>\theta$, we then infer
\begin{align*}
\Prob(\wh{X}_{n}(\gamma)\le 1)\ &\ge\ \Prob\left(\max_{1\le k\le n}\gamma^{k-1}Q_{k}\le\frac{1}{n}\right)\\
&\ge\ \Prob(Q=0)^{m_{n}}\prod_{m_{n}+1\le k\le n}\Prob(\log Q\le g(k,n))\\
&\ge\ \delta^{m_{n}}\prod_{m_{n}+1\le k\le n}\left(1-\frac{(1-\eps)\theta}{g(k,n)}\right)\\
&\ge\ (e^m n)^{-\eps}\prod_{k=m}^{n}\left(1-\frac{1-\eps}{k}\right)\\
&\ge\ (e^m n)^{-\eps}\prod_{k=2}^{n}\left(1-\frac{1}{k}\right)^{1-\eps}\ =\ \frac{e^{-m\varepsilon}}{n}.
\end{align*}
Here $(1-x)^{1-\eps}\le 1-(1-\eps)x$ for all $x\in [0,1]$ has been utilized for the last inequality. Hence, $\sum_{n\ge 1}\Prob(\wh{X}_{n}(\gamma)\le 1)=\infty$, giving the recurrence of $(X_{n}(\gamma))_{n\ge 0}$ by Proposition \ref{recur}.\qed
\end{proof}

Given a Markov chain $(Z_{n})_{n\ge 0}$, a sequence $(\sigma_{n})_{n\ge 0}$ is called a \emph{renewal stopping sequence} for this chain if the following conditions hold:
\begin{description}[(R2)]
\item[(R1)] $\sigma_{0}=0$ and the $\tau_{n}:=\sigma_{n}-\sigma_{n-1}$ are iid for $n\ge 1$.
\vspace{.1cm}
\item[(R2)] There exists a filtration $\cF=(\cF_{n})_{n\ge 0}$ such that $(Z_{n})_{n\ge 0}$ is Markov-adapted and each $\sigma_{n}$ is a stopping time with respect to $\cF$.
\end{description}

We define
$$ S_{n}\ :=\ \log\Pi_{n}\ =\ \sum_{k=1}^{n}\log M_{k} $$
for $n\ge 0$ and recall that, by our standing assumption, $(S_{n})_{n\ge 0}$ is a negative divergent random walk $(S_{n}\to-\infty$ a.s.). For $c\in\R$, let $(\sgn(c))_{n\ge 0}$ and $(\sln(c))_{n\ge 0}$ denote the possibly defective renewal sequences of ascending and descending ladder epochs associated with the random walk $(S_{n}+cn)_{n\ge 0}$, in particular
\begin{align*}
\sg(c)\ &=\ \sg_{1}(c)\ :=\ \inf\{n\ge 1:S_{n}+cn>0\}\ =\ \inf\{n\ge 1:\Pi_{n}>e^{-cn}\},\\
\sl(c)\ &=\ \sl_{1}(c)\ :=\ \inf\{n\ge 1:S_{n}+cn<0\}\ =\ \inf\{n\ge 1:\Pi_{n}<e^{-cn}\}.
\end{align*}
Plainly, these are renewal stopping sequences for $(X_{n})_{n\ge 0}$ whenever nondefective.

\begin{Lemma}\label{lem:bounding lemma}
Let $c\ge 0$ and $\gamma=e^{-c}$.
\begin{description}[(b)]
\item[(a)] If $c$ is such that $\sl(c)<\infty$ a.s., then, with $(\sigma_{n})_{n\ge 0}:=(\sln(c))_{n\ge 0}$,
$$ X_{\sigma_{n}}\ \le\ X_{\sigma_{n}}(\gamma)\quad\text{and}\quad\wh{X}_{\sigma_{n}}\ \le\ \wh{Y}_{n}\quad\text{a.s.} $$
for all $n\ge 0$, where $X_{0}=X_{0}(\gamma)=\wh{Y}_{0}$ and
$$ \wh{Y}_{n}\ :=\ \sum_{k=1}^{n}\gamma^{\sigma_{k-1}}Q_{k}^{*}\quad\text{with}\quad Q_{n}^{*}\ :=\ \sum_{k=1}^{\sigma_{n}-\sigma_{n-1}}\frac{\Pi_{\sigma_{n-1}+k-1}}{\Pi_{\sigma_{n-1}}}Q_{k} $$
for $n\ge 1$ denotes the sequence of backward iterations pertaining to the recursive Markov chain $Y_{n}=\gamma^{\sigma_{n}-\sigma_{n-1}}Y_{n-1}+Q_{n}^{*}$.
\vspace{.1cm}
\item[(b)] If $c$ is such that $\sg(c)<\infty$ a.s., then, with $(\sigma_{n})_{n\ge 0}:=(\sgn(c))_{n\ge 0}$,
$$ X_{\sigma_{n}}\ \ge\ X_{\sigma_{n}}(\gamma)\quad\text{and}\quad\wh{X}_{\sigma_{n}}\ \ge\ \wh{Y}_{n}\quad\text{a.s.} $$
for all $n\ge 0$, where $X_{0}=X_{0}(\gamma)=\wh{Y}_{0}$ and $\wh{Y}_{n}$ is defined as in (a) for the $\sigma_{n}$ given here.
\end{description}
\end{Lemma}

Plainly, one can take $c\in (0,-\fm)$ in (a) and $c\in (-\fm,\infty)$ in (b) if $-\infty<\fm<0$.

\begin{proof}
(a) Suppose that the $\sln(c)$ are a.s. finite. To prove our claim for $X_{\sigma_{n}}$, we use induction over $n$. Since $\sigma_{0}=0$, we have $X_{\sigma_{0}}=X_{\sigma_{0}}(\gamma)$. For the inductive step suppose that $X_{\sigma_{n-1}}\le X_{\sigma_{n-1}}(\gamma)$ for some $n\ge 1$. Observe that, with $\tau_{n}=\sigma_{n}-\sigma_{n-1}$,
\begin{align}
\begin{split}\label{eq:crucial estimate}
M_{\sigma_{n-1}+k+1}\cdot\ldots\cdot M_{\sigma_{n}}\ &=\ \frac{\Pi_{\sigma_{n}}}{\Pi_{\sigma_{n-1}+k}}\\
&=\ e^{(S_{\sigma_{n}}+c\sigma_{n})-(S_{\sigma_{n-1}+k}+c(\sigma_{n-1}+k))-c(\tau_{n}-k)}\ \le\ \gamma^{\tau_{n}-k}
\end{split}
\end{align}
for all $0\le k\le\tau_{n}$. Using this and the inductive hypothesis, we obtain
\begin{align*}
X_{\sigma_{n}}\ &=\ \frac{\Pi_{\sigma_{n}}}{\Pi_{\sigma_{n-1}}}X_{\sigma_{n-1}}+\sum_{k=1}^{\tau_{n}}\frac{\Pi_{\sigma_{n}}}{\Pi_{\sigma_{n-1}+k}}\,Q_{\sigma_{n-1}+k}\\
&\le\ \gamma^{\tau_{n}}X_{\sigma_{n-1}}(\gamma)+\sum_{k=1}^{\tau_{n}}\gamma^{\tau_{n}-k}\,Q_{\sigma_{n-1}+k}\ =\ X_{\sigma_{n}}(\gamma)\quad\text{a.s.}
\end{align*}
as asserted. Regarding the backward iteration $\wh{X}_{\sigma_{n}}$, we find more directly that
\begin{align*}
\wh{X}_{\sigma_{n}}\ &=\ \sum_{k=1}^{n}\Pi_{\sigma_{k-1}}\sum_{j=1}^{\tau_{k}}\frac{\Pi_{\sigma_{k-1}+j-1}}{\Pi_{\sigma_{k-1}}}Q_{j}\\
&\le\ \sum_{k=1}^{n}\gamma^{\sigma_{k-1}}\sum_{j=1}^{\tau_{k}}\frac{\Pi_{\sigma_{k-1}+j-1}}{\Pi_{\sigma_{k-1}}}Q_{j}\ =\ \wh{Y}_{n}\quad\text{a.s.}
\end{align*}
for each $n\ge 1$.

\vspace{.1cm}
(b) If $c$ is such that the $\sgn(c)$ are a.s. finite, then \eqref{eq:crucial estimate} turns into
\begin{equation*}
M_{\sigma_{n-1}+k+1}\cdot\ldots\cdot M_{\sigma_{n}}\ \ge\ \gamma^{\tau_{n}-k}
\end{equation*}
for all $n\in\N$ and $0\le k\le\tau_{n}$. Now it is easily seen that the inductive argument in (a) remains valid when reversing inequality signs and the same holds true for $\wh{X}_{\sigma_{n}}$.\qed
\end{proof}

\section{Tail lemmata}\label{sec:tail lemma}

In order to prove our results, we need to verify that the tail condition \eqref{tail2} is preserved under stopping times with finite mean. To be more precise, let $\sigma$ be any such stopping time for $(M_{k},Q_{k})_{k\ge 1}$ and consider
$$ \wh{X}_{\sigma}\ =\ \sum_{k=1}^{\sigma}\Pi_{k-1}Q_{k}. $$
Obviously,
\begin{equation}\label{Qsigma* bounds}
\max_{1\le k\le \sigma}\,\Pi_{k-1}Q_{k}\ \le\ \wh{X}_{\sigma}\ \le\ \sigma\,\max_{1\le k\le \sigma}\,\Pi_{k-1}Q_{k}.
\end{equation}

\begin{Lemma}\label{tail lemma}
Assuming \eqref{trivial}, \eqref{trivial2} and $\fm<0$, condition \eqref{tail2} entails
\begin{equation*}
\lim_{t\to\infty}t\,\Prob(\log \wh{X}_{\sigma}>t)\ =\ s\,\Erw\sigma,
\end{equation*}
where the right-hand side equals $0$ if $s=0$, and $\infty$ if $s=\infty$.
\end{Lemma}

\begin{proof}
It suffices to prove
\begin{equation*}
\lim_{t\to\infty}t\,\Prob\left(\log\max_{1\le k\le \sigma}\,\Pi_{k-1}Q_{k}>t\right)\ =\ s\,\Erw \sigma
\end{equation*}
because \eqref{Qsigma* bounds} in combination with $\Erw\sigma<\infty$ entails
\begin{align*}
\Prob&\left(\log\max_{1\le k\le \sigma}\,\Pi_{k-1}Q_{k}>t\right)\ \le\ \Prob(\log \wh{X}_{\sigma}>t)\\
&\le\ \Prob(\log\sigma>\eps t)\ +\ \Prob\left(\log\max_{1\le k\le \sigma}\,\Pi_{k-1}Q_{k}>(1-\eps)t\right)\\
&=\ o(t^{-1})\ +\ \Prob\left(\log\max_{1\le k\le \sigma}\,\Pi_{k-1}Q_{k}>(1-\eps)t\right)
\end{align*}
for all $\eps>0$.

\vspace{.1cm}
(a) We first prove that
\begin{equation}\label{eq:upper bound}
\limsup_{t\to\infty}\,t\,\Prob\left(\log\max_{1\le k\le \sigma}\,\Pi_{k-1}Q_{k}>t\right)\ \le\ s\,\Erw\sigma
\end{equation}
which is nontrivial only when assuming $s\in [0,\infty)$. Put $\eta_{n}:=\log Q_{n}$ for $n\in\N$.
For any $\eps\in (0,1)$, we then have
\begin{align*}
\Prob&\left(\log\max_{1\le k\le \sigma}\Pi_{k-1}Q_{k}>t\right)\ =\ \Prob\left(\max_{1\le k\le \sigma}\,(S_{k-1}+\eta_{k})>t\right)\\
&\le\ \Prob\left(\max_{0\le k\le \sigma}\,S_{k}>\eps t\right)\ +\ \Prob\left(\max_{1\le k\le \sigma}\,\eta_{k}>(1-\eps)t\right)\\
&=\ I_{1}(t)\ +\ I_{2}(t).
\end{align*}
Regarding $I_{1}(t)$, notice that
$$ \max_{0\le k\le\sigma}S_{k}\ \le\ \sum_{k=1}^{\sigma}\log_{+}M_{k}. $$
Since $\fm\in [-\infty,0)$ entails $\Erw\log_{+}M<\infty$ and thus, by Wald's identity,
$$ \Erw\left(\max_{0\le k\le\sigma}S_{k}\right)\ \le\ \Erw\left(\sum_{k=1}^{\sigma}\log_{+}M_{k}\right)\ =\ \Erw\sigma\,\Erw\log_{+}M\ <\ \infty. $$
As a consequence,
$$ \lim_{t\to\infty}t\,I_{1}(t)\ =\ 0. $$

Turning to $I_2(t)$, we obtain
\begin{align*}
t\,I_{2}(t)\ &\le\ t\,\Erw\sum_{k=1}^{\sigma}\1_{\{\eta_{k}>(1-\eps)t\}}\\
&=\ t\,\Erw\sum_{k\ge 1}\1_{\{\eta_{k}>(1-\eps)t,\,\sigma\ge k\}}\\
&=\ t\,\Prob(\eta_{1}>(1-\eps)t)\sum_{k\ge 1}\Prob(\sigma\ge k)\\
&=\ t\,\Erw\sigma\,\Prob(\eta_{1}>(1-\eps)t)\ <\ \infty
\end{align*}
and thereupon
$$ \limsup_{t\to\infty}\,t\,(I_{1}(t)+I_{2}(t))\ =\ \limsup_{t\to\infty}\,t\,I_{2}(t)\ \le\ \frac{s\,\Erw\sigma}{1-\eps}. $$
Hence \eqref{eq:upper bound} follows upon letting $\eps$ tend to 0.

\vspace{.1cm}
(b) It remains to show the inequality
\begin{equation}\label{eq:lower bound}
\liminf_{t\to\infty}\,t\,\Prob\left(\log\max_{1\le k\le \sigma}\,\Pi_{k-1}Q_{k}>t\right)\ \ge\ s\,\Erw\sigma
\end{equation}
which is nontrivial only when assuming $s\in (0,\infty]$. To this end observe that
$$ \log\max_{1\le k\le \sigma}\,\Pi_{k-1}Q_{k}\ =\ \max_{1\le k\le\sigma}(S_{k-1}+\eta_{k})\ \ge\ \max_{1\le k\le\sigma\wedge\tau(c)}\eta_{k}-c $$
for any $c>0$, where $\tau(c):=\inf\{n\ge 1:S_{n}<-c\}$. Since, furthermore,
\begin{align*}
\Prob\left(\max_{1\le k\le \sigma\wedge\tau(c)}\,\eta_{k}>t\right)\ &=\ \Erw\left(\sum_{k=1}^{\sigma\wedge\tau(c)}\1_{\{\eta_{1}\vee...\eta_{k-1}\le t,\eta_{k}>t\}}\right)\\
&=\ \sum_{k\ge 1}\Prob\left(\max_{1\le j\le k-1}\eta_{j}\le t,\,\eta_{k}>t,\,\sigma\wedge\tau(c)\ge k\right)\\
&=\ \Prob(\eta_{1}>t)\sum_{k\ge 1}\Prob\left(\max_{1\le j\le k-1}\eta_{j}\le t,\,\sigma\wedge\tau(c)\ge k\right),
\end{align*}
we find
\begin{align*}
t\,&\Prob\left(\log\max_{1\le k\le \sigma}\,\Pi_{k-1}Q_{k}>t\right)\\
&\ge\ t\,\Prob(\eta_{1}>t+c)\sum_{k\ge 1}\Prob\left(\max_{1\le j\le k-1}\eta_{j}\le t,\,\sigma\wedge\tau(c)\ge k\right)\\
&\underset{t\to\infty}{\longrightarrow}\ s\,\Erw(\sigma\wedge\tau(c)),
\end{align*}
and this implies \eqref{eq:lower bound} upon letting $c$ tend to $\infty$, for $\sigma\wedge\tau(c)\uparrow\sigma$.\qed
\end{proof}

By combining the previous result with a simple stochastic majorization argument, we obtain the following extension.

\begin{Lemma}\label{tail lemma 2}
Let $s_{*}$ and $s^{*}$ be as defined in \eqref{eq:def of s_* and s^*}. Then
\begin{align}
\limsup_{t\to\infty}\,t\,\Prob(\log \wh{X}_{\sigma}>t)\ &\le\ s^{*}\,\Erw\sigma\label{565+}\\
\text{and}\quad\liminf_{t\to\infty}\,t\,\Prob(\log \wh{X}_{\sigma}>t)\ &\ge\ s_{*}\,\Erw\sigma.\label{565-}
\end{align}
\end{Lemma}

\begin{proof}
For \eqref{565+}, we may assume $s^{*}<\infty$. Recall the notation $F(t)=\Prob(\log Q\le t)$ and put $\ovl{F}:=1-F$. Then define the new distribution function $G$ by
$$ \ovl{G}(t)\ :=\ \1_{(-\infty,0]}(t)+\left(\ovl{F}(t)\vee\frac{s}{s+t}\right)\1_{(0,\infty)}(t) $$
for some arbitrary $s>s^{*}\ ($we can even choose $s=s^{*}$ unless $s^{*}=0)$. Since $\ovl{G}\ge\ovl{F}$, we may construct (on a possibly enlarged probability space) random variables $Q',\,Q_{1}',\,Q_{2}',\ldots$ such that $(M,Q,Q'),\,(M_{1},Q_{1},Q_{1}'),\,(M_{2},Q_{2},Q_{2}'),\ldots$ are iid, the distribution function of $\log Q'$ is $G$, and $Q'\ge Q$, thus
$$ \wh{X}_{\sigma}'\ :=\ \sum_{k=1}^{\sigma}\Pi_{k-1}Q_{k}'\ \ge\ \wh{X}_{\sigma}. $$
On the other hand, $\ovl{G}(t)=\Prob(\log Q'>t)$ satisfies the tail condition \eqref{tail2}, whence, by an appeal to Lemma \ref{tail lemma},
$$ \limsup_{t\to\infty}\,t\,\Prob(\log \wh{X}_{\sigma}>t)\ \le\ \lim_{t\to\infty}\,t\,\Prob(\log \wh{X}_{\sigma}'>t)\ =\ s\,\Erw\sigma. $$
This proves \eqref{565+} because $s-s^{*}$ can be chosen arbitrarily small.

\vspace{.1cm}
Assertion \eqref{565-} for $s>0$ is proved in a similar manner. Indeed, pick any $s\in (0,s_{*})\ ($or even $s_{*}$ itself unless $s_{*}=\infty)$ and define
$$ \ovl{G}(t)\ :=\ \left(\ovl{F}(t)\wedge\frac{s}{s+t}\right)\1_{[0,\infty)}(t) $$
which obviously satisfies $\ovl{G}\le\ovl{F}$. In the notation from before, we now have $Q'\le Q$ and thus $\wh{X}_{\sigma}'\le \wh{X}_{\sigma}$. Since again $\ovl{G}(t)=\Prob(\log Q'>t)$ satisfies the tail condition \eqref{tail2}, we easily arrive at the desired conclusion by another appeal to Lemma \ref{tail lemma}.\qed
\end{proof}

Our last tail lemma will be crucial for the proof of Theorem \ref{main12}. Given any $0<\gamma<1$, recall that $X_{0}(\gamma)=X_{0}$ and
$$ X_{n}(\gamma)\ =\ \gamma X_{n-1}(\gamma)+Q_{n} $$
for $n\ge 1$. Let $\sigma$ be any integrable stopping time for $(X_{n})_{n\ge 0}$ and note that
$$ X_{\sigma}(\gamma)\ =\ \gamma^{\sigma}X_{0}+Q(\gamma), $$
where
$$ Q(\gamma)\ :=\ \sum_{k=1}^{\sigma}\gamma^{\sigma-k}Q_{k}. $$
More generally, if $(\sigma_{n})_{n\ge 0}$ denotes a renewal stopping sequence for $(X_{n})_{n\ge 0}$ with $\sigma=\sigma_{1}$, then
$$ X_{\sigma_{n}}(\gamma)\ =\ \gamma^{\sigma_{n}-\sigma_{n-1}}X_{\sigma_{n-1}}(\gamma)+Q_{n}(\gamma) $$
for $n\ge 1$ with iid $(\gamma^{\sigma_{n}-\sigma_{n-1}},Q_{n}(\gamma))_{n\ge 1}$ and $Q_{\sigma_{1}}(\gamma)=Q(\gamma)$.

\begin{Lemma}\label{tail lemma 3}
Let $\gamma\in (0,1)$ and $\sigma,Q(\gamma)$ be as just introduced. If $Q$ satisfies condition \eqref{tail2}, then
\begin{equation*}
\lim_{t\to\infty}t\,\Prob(\log Q(\gamma)>t)\ =\ s\,\Erw\sigma,
\end{equation*}
where the right-hand side equals $0$ if $s=0$, and $\infty$ if $s=\infty$. More generally, with $s_{*},s^{*}$ as defined in \eqref{eq:def of s_* and s^*}, it is always true that
\begin{align*}
\begin{split}
s_{*}\,\Erw\sigma\ &\le\ \liminf_{t\to\infty}t\,\Prob(\log Q(\gamma)>t)\\
&\le\ \limsup_{t\to\infty}t\,\Prob(\log Q(\gamma)>t)\ \le\ s^{*}\,\Erw\sigma.
\end{split}
\end{align*}
\end{Lemma}

\begin{proof}
Embarking on the obvious inequality (compare \eqref{Qsigma* bounds})
$$ \max_{1\le k\le\sigma}\gamma^{\sigma-k}Q_{k}\ \le\ Q(\gamma)\ \le\ \sigma\max_{1\le k\le\sigma}\gamma^{\sigma-k}Q_{k}, $$
the arguments are essentially the same and even slightly simpler than those given for the proofs of Lemmata \ref{tail lemma} and \ref{tail lemma 2}. We therefore omit further details.\qed
\end{proof}

\section{Proof of Theorems \ref{main11} and \ref{main12}}\label{sec:main11 and main12}

\begin{proof}[of Theorem \ref{main11}]
(a) \emph{Null recurrence}: We keep the notation of the previous sections, in particular $S_{n}=\log\Pi_{n}$ and $\eta_{n}=\log Q_{n}$ for $n\ge 1$. For an arbitrary $c>0$, let $(\sigma_{n})_{n\ge 0}$ be the integrable renewal stopping sequence with
$$ \sigma\ =\ \sigma_{1}\ :=\ \inf\{n\ge 1:S_{n}<-c\}. $$
Then
$$ \big(M_{n}^{*},Q_{n}^{*}\big)\ :=\ \left(\frac{\Pi_{\sigma_{n}}}{\Pi_{\sigma_{n-1}}},\sum_{k=\sigma_{n-1}+1}^{\sigma_{n}}\frac{\Pi_{k-1}}{\Pi_{\sigma_{n-1}}}Q_{k}\right),\quad n\ge 1, $$
are independent copies of $(\Pi_{\sigma},\sum_{k=1}^{\sigma}\Pi_{k-1}Q_{k})$. Put
$$ \Pi_{0}^{*}\ :=\ 1\quad\text{and}\quad\Pi_{n}^{*}\ :=\ \prod_{k=1}^{n}M_{k}^{*}\quad\text{for }n\ge 1. $$
By Lemma \ref{tail lemma 2},
$$ \limsup_{t\to\infty}\,t\,\Prob(\log Q_{1}^{*}>t)\ \le\ s^{*}\,\Erw\sigma $$

As already pointed out in the Introduction, validity of \eqref{trivial}, \eqref{trivial2} and \eqref{30} implies that $(X_{n})_{n\ge 0}$ cannot be positive recurrent. We will always assume $X_{0}=\wh{X}_{0}=0$ hereafter. By Proposition \ref{recur}, the null recurrence of $(X_{n})_{n\ge 0}$ follows if we can show that
\begin{equation}\label{eq:to show}
\sum_{n\ge 1}\Prob(X_{n}\le t)\ =\ \sum_{n\ge 1}\Prob(\wh{X}_{n}\le t)\ =\ \infty
\end{equation}
for some $t>0$ or, a fortiori,
\begin{equation}\label{eq2:to show}
\sum_{n\ge 1}\Prob(\wh{X}_{\sigma_{n}}\le t)\ =\ \infty.
\end{equation}
We note that $\wh{X}_{\sigma_{n}}=\sum_{k=1}^{\sigma_{n}}\Pi_{k-1}Q_{k}=\sum_{k=1}^{n}\Pi_{k-1}^{*}Q_{k}^{*}$ and pick an arbitrary nondecreasing sequence $0=a_{0}\le a_{1}\le\ldots$ such that
$$ a\ :=\ \sum_{n\ge 0}e^{-a_{n}}\ <\ \infty. $$
Fix any $z>0$ so large that
\begin{equation*}
\Prob\left(Q_{1}^{*}\le\frac{z}{a}\right)\ >\ 0.
\end{equation*}
Using $M_{n}^{*}<1$ for all $n\ge 1$, we then infer that
\begin{align*}
\Prob\left(\max_{1\le k\le n}e^{a_{k-1}}\Pi_{k-1}^{*}Q_{k}^{*}\le\frac{t}{a}\right)\ &\ge\ \Prob\left(\max_{1\le k\le n}\Pi_{k-1}^{*}Q_{k}^{*}\le\frac{t}{a}\right)\\
&\ge\ \Prob\left(Q_{1}^{*}\le\frac{t}{a}\right)^{n}\ >\ 0.
\end{align*}
Furthermore,
\begin{equation*}
\Prob(\wh{X}_{\sigma_{n}}\le t)\ =\ \Prob\left(\sum_{k=1}^{n}\Pi_{k-1}^{*}Q_{k}^{*}\le t\right)\ \ge\ \Prob\left(\max_{1\le k\le n}e^{\,a_{k-1}}\Pi_{k-1}^{*}Q_{k}^{*}\le\frac{t}{a}\right),
\end{equation*}
because $\max_{1\le k\le n}e^{\,a_{k-1}}\Pi_{k-1}^{*}Q_{k}^{*}\le\frac{t}{a}$ implies
$$ \sum_{k=1}^{n}\Pi_{k-1}^{*}Q_{k}^{*}\ \le\ \frac{t}{a}\sum_{k=1}^{n}e^{-a_{k-1}}\ \le\ t. $$
Consequently,
\begin{equation}\label{eq3:to show}
\sum_{n\ge 1}\Prob\left(\max_{1\le k\le n}e^{a_{k-1}}\Pi_{k-1}^{*}Q_{k}^{*}\le\frac{t}{a}\right)\ =\ \infty
\end{equation}
implies \eqref{eq2:to show}, and thus \eqref{eq:to show}.

\vspace{.1cm}
By choice of the $\sigma_{n}$, we have $\log\Pi_{k}^{*}\le -ck$ a.s. Putting $x=\log t-\log a$, we have with $a_{k}=o(k)$ as $k\to\infty\ ($choose e.g. $a_{k}=2\,\log(1+k))$
\begin{align*}
\sum_{n\ge 1}\,&\Prob\left(\max_{1\le k\le n}e^{a_{k-1}}\Pi_{k-1}^{*}Q_{k}^{*}\le\frac{t}{a}\right)\\
&\ge\ \sum_{n\ge 1}\Prob\left(\max_{1\le k\le n}\big(-c(k-1)+a_{k-1}+\log Q_{k}^{*}\big)\le x\right)\\
&=\ \sum_{n\ge 1}\prod_{k=1}^{n}\Prob\big(\log Q_{1}^{*}\le x-a_{k-1}+c(k-1)\big)
\end{align*}
Defining $b_{n}$ as the $n$th summand in the previous sum and writing $\sigma=\sigma(c)$ to show the dependence on $c$, Lemma \ref{tail lemma 2} provides us with
$$ \liminf_{n\to\infty}\,n\left(\frac{b_{n+1}}{b_{n}}-1\right)\ \ge\ -s^{*}\,\frac{\Erw\sigma(c)}{c}, $$
hence Raabe's test entails \eqref{eq3:to show} if we can fix $c>0$ such that
\begin{equation}\label{eq4:to show}
s^{*}\,\frac{\Erw\sigma(c)}{c}\ <\ 1.
\end{equation}
Plainly, the latter holds true for any $c>0$ if $s^{*}=0$. But if $s^{*}\in (0,\infty)$, then use the elementary renewal theorem to infer (also in the case $\fm=-\infty$)
$$ \lim_{c\to\infty}\,s^{*}\,\frac{\Erw\sigma(c)}{c}\ =\ \frac{s^{*}}{-\fm}\ <\ 1. $$
Hence, \eqref{eq4:to show} follows by our assumption $s^{*}<-\fm$.\qed

\vspace{.2cm}\noindent
(b) \emph{Transience}: By Proposition \ref{transient}, it must be shown that
$$ \sum_{n\ge 0}\Prob(\wh{X}_{n}\le t)\ <\ \infty $$
for any $t>0$. We point out first that it suffices to show
\begin{equation}\label{reduced sum finite}
\sum_{n\ge 0}\Prob(\wh{X}_{\sigma_{n}}\le t)\ <\ \infty
\end{equation}
for some integrable renewal stopping sequence $(\sigma_{n})_{n\ge 0}$. Namely, since $(\wh{X}_{n})_{n\ge 0}$ is nondecreasing, it follows that
\begin{align*}
\sum_{n\ge 0}\Prob(\wh{X}_{n}\le t)\ &=\ \sum_{n\ge 0}\Erw\left(\sum_{k=\sigma_{n}}^{\sigma_{n+1}-1}\1_{\{\wh{X}_{k}\le t\}}\right)\\
&\le\ \sum_{n\ge 0}\Erw\big(\sigma_{n+1}-\sigma_{n}\big)\1_{\{\wh{X}_{\sigma_{n}}\le t\}}\\
&=\ \Erw\sigma\sum_{n\ge 0}\Prob(\wh{X}_{\sigma_{n}}\le t),
\end{align*}
where we have used that $\sigma_{n+1}-\sigma_{n}$ is independent of $\wh{X}_{\sigma_{n}}$ for each $n\ge 0$.

\vspace{.1cm}
Choosing $(\sigma_{n})_{n\ge 0}=(\sgn(c))_{n\ge 0}$ as defined before Lemma \ref{lem:bounding lemma} for an arbitrary $c\in (-\fm,s_{*})$, part (b) of this lemma provides us with
\begin{align}\label{wh(X)_{n}>=wh(Y)_{n}}
\wh{X}_{\sigma_{n}}\ \ge\ \wh{Y}_{n}\ =\ \sum_{k=1}^{n}\gamma^{\sigma_{k-1}}Q_{k}^{*}\quad\text{a.s.}
\end{align}
for all $n\ge 0$, where the $Q_{n}^{*}$ are formally defined as in (a) for the $\sigma_{n}$ given here and the $\wh{Y}_{n}$ are the backward iterations of the Markov chain defined by the RDE
$$ Y_{n}\ =\ \gamma^{\sigma_{n}-\sigma_{n-1}}Y_{n-1}+Q_{n}^{*},\quad n\ge 1. $$
By Lemma \ref{tail lemma 2},
$$ \liminf_{t\to\infty}t\,\Prob(\log Q^{*}>t)\ \ge\ s_{*}\Erw\sigma. $$
Let $(Q_{n}')_{n\ge 1}$ be a further sequence of iid random variables with generic copy $Q'$, independent of all other occurring random variables and such that
\begin{equation}\label{eq:tail Q'}
\lim_{t\to\infty}t\,\Prob(\log Q'>t)\ =:\ s\ \in\ (c,s_{*}).
\end{equation}
Put $\gamma:=e^{-c}$. Then Kellerer's result (Proposition \ref{Kellerer's result}) implies the transience of the Markov chain $X_{n}'(\gamma)=\gamma X_{n-1}'(\gamma)+Q_{n}'$, $n\ge 1$, and thus also of the subchain $(X_{\sigma_{n}}'(\gamma))_{n\ge 0}$. Since $\wh{X}_{\sigma_{n}}'(\gamma)\ =\ \sum_{k=1}^{n}\gamma^{\sigma_{k-1}}\wh{Q}_{k}$ with
$$ \wh{Q}_{n}\ =\ \sum_{k=\sigma_{n-1}+1}^{\sigma_{n}}\gamma^{k-\sigma_{n-1}-1}Q_{k}'  $$
for $n\ge 1$ and since, by \eqref{eq:tail Q'} and Lemma \ref{tail lemma},
$$ \lim_{t\to\infty}t\,\Prob(\log\wh{Q}>t)\ =\ s\,\Erw\sigma, $$
thus $\Prob(Q^{*}>t)\ge\Prob(\wh{Q}>t)$ for all sufficiently large $t$, we now infer by invoking our Comparison Lemma \ref{comparison lemma} that the transience of $(X_{\sigma_{n}}'(\gamma))_{n\ge 0}$ entails the transience of $(Y_{n})_{n\ge 0}$ given above and thus
$$ \sum_{n\ge 0}\Prob(\wh{Y}_{n}\le t)\ <\ \infty $$
for all $t>0$. Finally, use \eqref{wh(X)_{n}>=wh(Y)_{n}} to arrive at \eqref{reduced sum finite}.
This completes the proof of part (b).\qed
\end{proof}

\begin{proof}[of Theorem \ref{main12}]
Fix $c>s^{*}$ and put as before $\gamma=e^{-c}$. Since $S_{n}=\log\Pi_{n}\to-\infty$ a.s. and $\fm^{+}=\fm^{-}=\infty$, we have
$$ \lim_{n\to\infty}\frac{S_{n}}{n}\ =\ \lim_{n\to\infty}\frac{S_{n}+an}{n}\ =\ -\infty\quad\text{a.s. for all }a\in\R $$
due to Kesten's trichotomy (see e.g. \cite[p.~3]{KesMal:96}) and hence in particular $S_{n}+cn\to -\infty$ a.s. As a consequence, the sequence $(\sigma_{n})_{n\ge 0}=(\sln(c))_{n\ge 0}$ as defined before Lemma \ref{lem:bounding lemma} is an integrable renewal stopping sequence for $(X_{n})_{n\ge 0}$. Part (a) of this lemma implies
$$ X_{\sigma_{n}}\ \le\ X_{\sigma_{n}}(\gamma)\ =\ \gamma^{\sigma_{n}-\sigma_{n-1}}X_{\sigma_{n-1}}(\gamma)+Q_{n}(\gamma)\quad\text{a.s.} $$
for all $n\ge 0$, where $Q_{n}(\gamma)=\sum_{k=\sigma_{n-1}+1}^{\sigma_{n}}\gamma^{\sigma_{n}-k}Q_{k}$ for $n\ge 1$. Hence it is enough to prove the null recurrence of $(X_{\sigma_{n}}(\gamma))_{n\ge 0}$. To this end, note first that $\fm(\gamma):=\Erw\log\gamma^{\sigma_{1}}=-c\,\Erw\sigma_{1}\in (-\infty,0)$. Moreover, Lemma \ref{tail lemma 3}  provides us with
$$ \limsup_{t\to\infty}t\,\Prob(\log Q_{1}(\gamma)>t)\ \le\ s^{*}\Erw\sigma_{1}\ <\ c\,\Erw\sigma_{1}\ =\ -\fm(\gamma), $$
and so the null recurrence of $(X_{\sigma_{n}}(\gamma))_{n\ge 0}$ follows from Theorem \ref{main11}.\qed
\end{proof}

\section{On the structure of the attractor set}\label{attr}

The purpose of this section is to investigate the structure of the attractor set $L$ for the Markov chain $(X_{n})_{n\ge 0}$ defined by \eqref{chain}. Unlike before, we assume hereafter that $(X_{n})_{n\ge 0}$ is locally contractive and recurrent, the latter being an inevitable assumption for $L\ne\oslash$. To exclude the ``trivial case'' (as explained in the introduction) we assume $\Prob(M=0)=0$. Recall from the paragraph preceding Proposition \ref{recur} that $L$ consists of all accumulation points of $(X_{n}^{x}(\omega))_{n\ge 0}$ which turns out to be the same for all $x\in\R_{+}$ and $\Prob$-almost all $\omega$. As already mentioned in the Introduction, $(X_{n})_{n\ge 0}$ possesses a unique invariant distribution, say $\nu$, if \eqref{trivial2} and \eqref{33} hold. The attractor set then coincides with the support of $\nu$. In the positive recurrent case the structure of $L$ was analyzed in \cite{BurDamMik:16}. According to Theorem 2.5.5 from there, $L$ necessarily equals a half-line $[a,\infty)$ for some $a\ge 0$ if it is unbounded. If $L$ is bounded, no general results concerning local properties of $L$ are known. It may equally well be a fractal (for instance, a Cantor set) or an interval. Below we consider both the positive and null recurrent case. The second one is implied by hypotheses of Theorem \ref{main11} a), but also holds when $\Erw \log M = 0$ (see \cite{BabBouElie:97,Benda:98b} for more details).

\vspace{.1cm}
For $(m,q)\in\R_{+}^{2}$, let $g$ be the affine transformation of $\R$ defined by
\begin{equation*}
g(x)=mx+q\,,\quad x\in\R.
\end{equation*}
We will write $g=(m,q)$, thereby identifying $g$ with $(m,q)$. The affine transformations constitute a group ${\sf Aff}(\R )$ with identity $(1,0)$
and multiplication defined by
$$g_{1}g_{2}=(m_{1},q_{1})\,(m_{2},q_{2})=(m_{1}m_{2},q_{1}+m_{1}q_{2})$$ for $g_{i}=(m_{i}, q_{i})$, $i=1,2$. The inverse of $g=(m,q)$ is given by $g^{-1}=(m^{-1},-m^{-1}q)$.

\vspace{.1cm}
Assuming $m\ne1$, let $x_{0} = x_{0}(g)=q/(1-m)$ be the unique fixed point of $g$, that is the unique solution to the equation $g(x)=x$. Then
$$ g(x)\ =\ m \,(x-x_{0})+x_{0}\,,\quad x\in\R$$ and similarly
\begin{equation}\label{eq:x0}
g^{n}(x)\ =\ m^{n}x+q_{n}= m^{n}\, (x-x_{0}) + x_{0}\,,\quad x\in \R,\ n\ge 1,
\end{equation}
where $q_{n}=\sum _{i=0}^{n-1}m^i\,q$. Formula \eqref{eq:x0} tells us that, modulo $x_{0}$, the action of $g$ is either contractive or expanding depending on whether $m<1$ or $m>1$, respectively.

\vspace{.1cm}
We interpret $\mu$, the distribution of $(M,Q)$, as a probability measure on ${\sf Aff}(\R)$ hereafter and let ${\rm supp}\,\mu$ denote its support. Consider the subsemigroup $T$  of ${\sf Aff} (\R)$
generated by ${\rm supp }\ \mu $, i.e.
$$
T\ :=\ \{ g_{1}\cdot \ldots \cdot g_{n}: g_{i} \in \supp\,\mu,\ i=1,\ldots,n,\,n\ge 1\}\,,
$$
and let $\ov T$ be its closure. A set $S\subset \R$ is said to be {\em $\ov T$-invariant} if for every $g\in \ov T$ and $x\in S$, $g(x)=mx+q  \in S$. The following result was stated in a slightly different setting as Proposition 2.5.3 in \cite{BurDamMik:16} and can be proved by the same arguments after minor changes.

\begin{Lemma}\label{lem:attractor}
Let $(X_{n})_{n\ge 0}$ be locally contractive and recurrent. Then $L=\ov S_{0}$, where
$$ S_{0}\ :=\ \{(1-m)^{-1}q: g=(m,q)\in T \,, m<1\}. $$
Moreover, $L$ equals the smallest $\ov T$-invariant subset of $\R$.
\end{Lemma}

For positive recurrent $(X_{n})_{n\ge 0}$, we have already pointed out that $L$, if unbounded, must be
a half-line $[a,\infty)$ ($a\ge 0$). The subsequent theorem provides the extension of this fact to any locally contractive and recurrent $(X_{n})_{n\ge 0}$.

\begin{Theorem}
Let $(X_{n})_{n\ge 0}$ be locally contractive and recurrent with unbounded attractor set $L$. If $\Prob(M=0)=0$, then $L=[a,\infty)$ for some $a\ge 0$ .
\end{Theorem}

\begin{proof}
By Lemma \ref{lem:attractor}, the set $L$ is uniquely determined by ${\rm supp}\,\mu$ and does not depend on the values $\mu(A)$ for any particular sets $A$. Consequently, any modification of $\mu$ with the same support leaves $L$ invariant. We will use this observation and define a tilting $\wt\mu$ of $\mu$ of the form
$$ \wt\mu({\rm d}m,{\rm d}q)\ =\ f(m)h(q)\,\mu({\rm d}m,{\rm d}q) $$
for suitable positive functions $f,h$ such that, if $(M,Q)$ has law $\wt\mu$, then the corresponding Markov chain $(\wt X_{n})_{n\ge 0}$ is positive recurrent with unique invariant distribution $\wt\nu$. We thus conclude ${\rm supp}\,\wt\nu = L$ and thereupon the claim $L=[a,\infty)$ if $L$ is unbounded.

\vspace{.1cm}
Put
\begin{align*}
f(m)\ &:=\ \begin{cases}
\displaystyle\frac{c_{0}}{|\log m|},&\text{if }0<m<\displaystyle\frac{1}{e},\\[1.5mm]
\hfill c_{0},&\text{if }\displaystyle\frac{1}{e}\le m<1,\\[1.5mm]
\hfill c_{0}\,c_{1},&\text{if }1\le m<e,\\[1.5mm]
\hfill\displaystyle\frac{c_{0}\,c_{1}}{\log m},&\text{if }m\ge e,
\end{cases}
\shortintertext{and}
h(q)\ &:=\ \begin{cases} \hfill c_{2},&\text{if }0\le q<e,\\[1.5mm]
\displaystyle\frac{c_{2}}{\log q},&\text{if }q\ge e,
\end{cases}
\end{align*}
and fix $c_{0},c_{1},c_{2}>0$ such that
$$ \int f(m) h(q)\ \mu({\rm d}m,{\rm d}q)\ =\ 1. $$
Observe that, if $\Prob_{\wt\mu}$ is such that $(M,Q)$ has law $\wt\mu$ under this probability measure, then
\begin{align}\label{int}
\begin{split}
\Erw_{\wt\mu}\log M\ &=\ c_{0}\left[-\int_{(0,1]\times\R_{+}}(1\wedge|\log m|)\, h(q)\ \mu({\rm d}m,{\rm d}q)\right.\\
&\hspace{2cm}+\ c_{1}\left.\int_{(1,\infty)\times\R_{+}}(1\wedge\log m)\,h(q)\ \mu({\rm d}m,{\rm d}q)\right],
\end{split}
\end{align}
and from this it is readily seen that we can specify $c_{0},c_{1}$ further so as to have
$$ \Erw_{\wt\mu}\log M\ <\ 0. $$
Regarding $\Erw_{\wt\mu}\log^{+}Q$, we find
\begin{align*}
\Erw_{\wt\mu} \log^{+}Q\ &=\ c_{2}\left[\int_{\R_{+}\times [1,e]}\log q\,f(m)\ \mu({\rm d}m,{\rm d}q)\ +\ \int_{\R_+\times (e,\infty)} f(m)\ \mu({\rm d}m,{\rm d}q)\right]\ <\ \infty.
\end{align*}
Hence, if $(M,Q)$ has law $\wt\mu$, then the corresponding Markov chain $(\wt X_{n})_{n\ge 0}$ defined by \eqref{chain} is indeed positive recurrent. This completes the proof of the theorem.\qed
\end{proof}

The next lemma provides some conditions on $\mu$ that are easily checked and sufficient for $L$ to be unbounded.

\begin{Lemma}\label{lem: 3.4}
If $\Prob(M=0)=0$, then each of the following conditions on the law of $(M,Q)$ implies that $L$ is unbounded.
\begin{description}[(C2)]\itemsep2pt
\item[(C1)] The law of $Q$ has unbounded support.
\item[(C2)] $\Prob(M>1)>0$ and $\Prob(Q=0)<1$.
\end{description}
\end{Lemma}

\begin{proof}
Assume first (C1), put $\beta:=\sup\{x: x\in L\}$ and recall from Lemma \ref{lem:attractor} that $L$ is invariant under the action of ${\rm supp}\,\mu$, i.e., if $(m,q)\in {\rm supp }\,\mu$ and $x\in L$, then $mx+q\in L$. In particular,
\begin{equation}\label{eq:5}
m\beta+q\,\le\,\beta\quad\text{ for any } (m,q)\in {\rm supp }\, \mu.
\end{equation}
Hence, if $\beta>0$, we have $\beta\ge m\beta + q \ge q$ and conclude $\beta=\infty$, for $q$ can be chosen arbitrarily large.

\vspace{.1cm}
Assuming now (C2), pick $g=(m,q)\in {\rm supp}\,\mu$ such that $m>1$. Notice that $x=x_{0}(g)=q/(1-m)$, the unique fixed point of $g$, is negative or zero because $m>1$. Since, under our hypothesis, the attractor set consists of at least two points, one can choose some positive $y\in L$.  Using \eqref{eq:x0}, we then infer
$$ g^{n}(y)\ =\ m^{n}(y-x)\,+\,x\ \to\ \infty $$
as $n\to\infty$ which completes the proof.\qed
\end{proof}

The assumptions of Lemma \ref{lem: 3.4} are not optimal. Even if $\Prob(M<1)=1$ and the support of the distribution of $Q$ is bounded, the attractor set may be unbounded, as demonstrated by the next lemma.

\begin{Lemma}
Assume that $\Prob(M<1)=1$. Then the attractor set $L$ is bounded if, and only if, the set
$$ S_{1}\ =\ \big\{x_{0}=x_{0}(g):g\in {\rm supp\,}\mu\big\} $$
is bounded or, equivalently, $Q/(1-M)$ is a.s. bounded.
\end{Lemma}

\begin{proof}
Assuming that $S_{1}$ is bounded, denote by $a$ and $b$ its infimum and supremum, respectively. Since the closed interval $[a,b]$ is obviously $\ov T$-invariant, it must contain $L$ by Lemma \ref{lem:attractor} which implies that $L$ is bounded.

If $S_{1}$ is unbounded, then $S_{1}\subset S_{0}$ implies that $S_{0}$ and thus also $L=\ov S_{0}$ is unbounded by another appeal to Lemma \ref{lem:attractor}. \qed
\end{proof}

Finally, we turn to the case when the attractor set $L$ is bounded. As already mentioned, the local structure of $L$ cannot generally be described precisely. If $\mu$ is supported by $(a,0)$ and $(a,1-a)$ for some $0<a<1/2$, then $L\subset [0,1]$ equals the Cantor set obtained by initially removing $(a,1-a)$ from $[0,1]$ and successive self-similar repetitions of this action for the remaining intervals (see also \cite[Remark 7]{BurtonRoesler:95}). So the Cantor ternary set is obtained if $a=1/3$. On the other hand, we have the following result.

\begin{Lemma}
For $\alpha,\beta<1$ with $\alpha+\beta\ge 1$ suppose that $(\alpha,q_{\alpha}), (\beta,q_{\beta})\in {\rm supp }\,\mu$ and further  $x_{\alpha}:=q_{\alpha}/(1-\alpha)\le q_{\beta}/(1-\beta) =:x_{\beta}$. Then the interval $[x_{\alpha},x_{\beta}]$ is contained in $L$.
\end{Lemma}

\begin{proof}
W.l.o.g. we assume that $x_{\alpha}=0$ and $x_{\beta}=1$ so that the points in ${\rm supp}\,\mu$ are $f_{\alpha}:=(\alpha,0)$ and $f_{\beta}:=(\beta,1-\beta)$ rather than $(\alpha,q_{\alpha}), (\beta,q_{\beta})$ and $[0,1]\subset L$ must be verified.

\vspace{.1cm}
Pick any $x\in (0,1)$. Let $U$ be the subsemigroup of ${\sf Aff}(\R)$ generated by $f_{\alpha}$ and $f_{\beta}$. To prove that $x\in L$, it is sufficient by Lemma \ref{lem:attractor} to find a sequence $(g_{n})_{n\ge 1}$ in $U$ such that $x$ is an accumulation point of $(g_{n}(0))_{n\ge 1}$.

\vspace{.1cm}
We construct this sequence inductively. Observe first that $\alpha+\beta \ge 1$ implies
$$ x\,\in\,(0,1)\,\subset\, [0,\alpha] \cup [1-\beta,1]. $$
If $x$ is an element of $[0,\alpha]$, take $g_{1} = f_{\alpha}$, otherwise take $g_{1}=f_{\beta}$. In both cases,
$$ x\,\in\, [g_{1}(0), g_{1}(1)]\quad\text{and} \quad |g_{1}(1) -g_{1}(0)|\,\le\,\alpha\vee\beta. $$
Assume we have found $g_{n}=(a_{n},b_{n})$ such that
$$ x\,\in\,[g_{n}(0), g_{n}(1)]\quad\text{and}\quad |g_{n}(1) -g_{n}(0)|\,=\, a_{n}\,\le\,\big(\alpha \vee \beta\big)^n. $$
Using again $\alpha+\beta\ge 1$, we have
\begin{align*}
x\ &\in\ [g_{n}(0), g_{n}(1)]\ =\ [b_{n}, a_{n} + b_{n}]\\
&\subset\ [b_{n}, \alpha a_{n} + b_{n}] \cup [(1-\beta)a_{n} + b_{n}, a_{n} + b_{n}]\\
&=\ [g_{n}f_{\alpha}(0), g_{n}f_{\alpha}(1)] \cup [g_{n}f_{\beta}(0), g_{n}f_{\beta}(1)].
\end{align*}
Thus $x$ must belong to one of these intervals. If $x\in [g_{n}f_{\alpha}(0), g_{n}f_{\alpha}(1)]$, put $g_{n+1}=g_{n} f_{\alpha}$, otherwise put $g_{n+1}=g_{n} f_{\beta}$. In both cases,
$$ x\,\in\,[g_{n+1}(0), g_{n+1}(1)]\quad\text{and} \quad |g_{n+1}(1)-g_{n+1}(0)|\,=\,a_{n} (\alpha \vee \beta)\,\le\,\big(\alpha \vee \beta\big)^{n+1}. $$
Hence, $x$ is indeed an accumulation point of the sequence $(g_{n}(0))_{n\ge 1}$ and therefore an element of $L$.\qed
\end{proof}

\vspace{1cm}
\footnotesize
\noindent   {\bf Acknowledgements.}
The authors wish to thank two anonymous referees for various helpful remarks that helped to improve the presentation and for bringing reference \cite{DenKorWach:16} to our attention.
G. Alsmeyer was partially supported by the Deutsche Forschungsgemeinschaft (SFB 878) "Geometry, Groups and Actions". D. Buraczewski was partially supported by the National Science Centre, Poland (Sonata Bis, grant number DEC-2014/14/E/ST1/00588).
Part of this work was done while A.~Iksanov was visiting M\"unster in January, February and July 2015, 2016. He gratefully acknowledges hospitality and financial support.

\def\cprime{$'$}

\end{document}